\documentclass[reqno,a4paper,12pt]{amsart}
\usepackage{amsmath,amsthm,amsfonts,amssymb, mathrsfs}
\usepackage{verbatim, graphicx, ifthen, enumitem}
\usepackage[T1]{fontenc}
\usepackage [applemac] {inputenc}

\allowdisplaybreaks

\let\oldmarginpar\marginpar
\renewcommand\marginpar[1]{\-\oldmarginpar[\raggedleft\footnotesize #1]%
{\raggedright\footnotesize #1}}
\newtheorem*{theo}{Theorem}

\newtheorem{theorem}{Theorem}
\newtheorem{lemma}{Lemma}

\newtheorem{proposition}{Proposition}
\newtheorem{corollary}{Corollary}
\theoremstyle{definition}
\newtheorem{definition}{Definition}
\newtheorem{example}{Example}

\theoremstyle{remark}
\newtheorem{remark}{Remark}

\newcommand{\Z}{\mathbb{Z}}
\newcommand{\D}{\mathbb{D}}
\newcommand{\abs}[1]{|#1|}
\newcommand{\Abs}[1]{\left|#1\right|}
\newcommand{\Norm}[1]{\left\|#1\right\|}
\newcommand{\inner}[2]{\left\langle #1|#2 \right\rangle}
\newcommand{\seq}[1]{(#1)}

\newcommand{\set}[1]{\big\{#1\big\}}
\newcommand{\Set}[1]{\left\{#1\right\}}

\newcommand{\N}{\mathbb{N}}
\newcommand{\R}{\mathbb{R}}
\newcommand{\T}{\mathbb{T}}
\newcommand{\C}{\mathbb{C}}
\newcommand{\PW}{\mathrm{PW}}

\newcommand{\Q}{\mathbb{Q}}

\newcommand{\Hp}{\mathscr{H}}

\def\T{\mathbb{T}}
\def\N{\mathbb{N}}
\def\Z{\mathbb{Z}}
\def\Q{\mathbb{Q}}
\def\R{\mathbb{R}}
\def\C{\mathbb{C}}

\def\1{\mathbf{1}}

\newcommand{\dif}{\mathrm{d}}
\newcommand{\e}{\mathrm{e}}
\newcommand{\im}{\mathrm{i}}
\newcommand{\norm}[1]{\|#1\|}

\renewcommand{\Re}{\operatorname{Re}}
\renewcommand{\Im}{\operatorname{Im}}

\begin{document}

 \title[Local properties of Hilbert spaces of Dirichlet series]{Local properties of Hilbert spaces of \\ Dirichlet series}
  
\author{Jan-Fredrik Olsen}
\address{Centre for Mathematical Sciences, Lund University, P.O. Box 118, SE-221 00 Lund, Sweden}
\email{janfreol@maths.lth.se}

\begin{abstract}
We show that the asymptotic behavior of  the partial sums of a sequence of positive numbers determine   the local behavior of the   Hilbert space of Dirichlet series defined using these as weights.
This extends  results recently obtained describing the local behavior of Dirichlet series with  square summable coefficients in terms of local integrability, boundary behavior, Carleson measures and interpolating sequences.
As these spaces can be identified with functions spaces  on the infinite-dimensional polydisk, this gives new results on
 the Dirichlet and Bergman spaces on the infinite dimensional polydisk, as well as the scale of Besov-Sobolev spaces containing the Drury-Arveson space on the infinite dimensional unit ball. 
We use both techniques from the theory of sampling in   Paley-Wiener spaces, and classical results from analytic number theory.
\end{abstract}

 \maketitle

\section{introduction}
The theory of Dirichlet series, i.e. functions of the form $f(s) = \sum_{n \in \N} a_n n^{-s}$ with $s= \sigma + \im t$ as the complex variable,   offers a bridge between number theory and analysis. Perhaps the most appealing example of the power of this connection is given by the tauberian approach to the classical prime number theorem. One way to state the prime number theorem is to say that the Chebyshev-type inequalities
\begin{equation} \label{chebyshev-type inequalities}
	A \frac{x}{(\log x)^\alpha} \leq 
	 \sum_{n \leq x} w_n \leq B \frac{x}{(\log x)^\alpha},  
\end{equation}
with   coefficients
\begin{equation*}
	 w_n =  \left\{  \begin{array}{cc}1 & n \; \text{is  a prime}\\ 0 & \text{otherwise} \end{array} \right.,
\end{equation*}
and $\alpha    = 1$, holds for any $A,B > 1$ as long as    $x>0$  is taken to be sufficiently large.
Originally due to Ikehara, the general idea of the tauberian approach is  to connect the function theoretic properties of the Riemann zeta function $\zeta(s) = \sum_{n \in \N} n^{-s}$ to the growth of these partial sums (see e.g. \cite[p. 245]{tenenbaum1995}). As is well-known, the properties of the Riemann zeta function is closely related to the behavior of the  prime numbers through the Euler product formula
\begin{equation*}
	\zeta(s) = \prod_{p \; \text{prime}} \frac{1}{1 - p^{-s}}.
\end{equation*}

We study the connection between the asymptotic behavior in terms of the inequalities \eqref{chebyshev-type inequalities}   for general sequences $(w_n)_{n \in \N}$ of non-negative numbers, and   local function theoretic properties of  the Hilbert spaces  
\begin{equation*}
	 \Hp_w = \left\{ \sum_{n\in \N} a_n n^{-s} :   \sum \frac{\abs{a_n}^2}{w_n} < \infty \right\}.
\end{equation*}
(By convention, if $w_n = 0$, we exclude the basis vector $n^{-s}$ from this definition.) 

The recent interest in the theory of these types of spaces began with a paper by Hedenmalm, Lindqvist and Seip \cite{hls1997}, where in particular the local behavior of functions in the Dirichlet-Hardy space $\Hp^2$, which corresponds to the choice  $w_n \equiv 1$, is studied.  By the Cauchy-Schwarz inequality, the space $\Hp^2$ is seen to consist of functions analytic on the half-plane $\C_{1/2} = \{ \Re s > 1/2\}$. The results  of this and later contributions \cite{bayart2002paper, konyagin_queffelec2002, olsen_saksman2010, olsen_seip2008}  can   be summarised as saying that locally $\Hp^2$   looks much like the classical Hardy space 
\begin{equation*}
	 H^2(\C_{1/2})  =  \left\{ f \; \text{analytic on} \; \C_{1/2} :   \sup_{\sigma > 1/2} \frac{1}{2\pi} \int_\R \abs{f(\sigma + \im t)}^2 < \infty \right\}.
\end{equation*}

One of the starting points of the function theory for the Dirichlet-Hardy space is a simple, but striking, local connection indicated by comparing reproducing kernels, i.e. functions $k_w$ such that $\inner{f}{k_w} = f(w)$ for all $f$ in the space, and points $w$ in the domain of definition. For the space $\Hp^2$, the reproducing kernel at $w \in \C_{1/2}$ is the translate $k_w(s) := \zeta(s+ \bar{w})$ of the Riemann zeta function. The Riemann zeta function  is known to be a meromorphic function with a single pole of residue one at $s=1$. This yields the formula
\begin{equation*}
	 k_w(s) = \frac{1}{s + \bar{w}-1}  + h(s + \bar{w}),	 
\end{equation*}
where $h$ is an entire function. This reveals that $k_w$ is an analytic perturbation of the reproducing kernel for $H^2(\C_{1/2})$, namely the Szeg\H o-kernel $k_w^S(s) = (s + \bar{w}-1)$.

The   following results strengthens this local connection. The first is \cite[Theorem 4.11]{hls1997}, which was found independently by Montgomery \cite[p. 140]{montgomery1994} in the context of analytic number theory.
 	\begin{theo}[Local embedding theorem \cite{hls1997, montgomery1994}] Given a bounded interval $I$, there exists  $C>0$, depending only on the length of $I$, such that for all $F \in \Hp^2$ we have  $\sup_{\sigma > 1/2}\int_I \abs{F(\sigma + \im t)}^2 \dif t \leq C \norm{F}_{\Hp^2}^2$.
	\end{theo}
	It is an immediate consequence of this theorem that if $F \in \Hp^2$ then $F(s)/s \in H^2(\C_{1/2})$. In particular, this implies that functions in $\Hp^2$ have non-tangential boundary values almost everywhere on the abscissa $\sigma = 1/2$. The second theorem we mention is in some sense dual to the previous one, and describes the space spanned by the boundary functions.
	\begin{theo}[Local boundary function property \cite{olsen_saksman2010}] Given a bounded interval $I$ and a     function $f \in H^2(\C_{1/2})$, there exists $F \in \Hp^2$ such that  
	$F-f$ has an analytic continuation across the segment $1/2 + \im I$.
	\end{theo}
	With these results in hand, it is not difficult to show that  a compactly supported   positive measure $\mu$ on $\C_{1/2}$ is a Carleson measure for $\Hp^2$ if and only if it is a Carleson measure for $H^2(\C_{1/2})$ (see also the proof of Theorem \ref{interpolation theorem}). Recall that if   $H$ is a Hilbert space of functions on $\C_{1/2}$, we say that a positive Borel measure is Carleson for $H$ if there exists $C>0$ such that for all $f \in H$ we have
	\begin{equation*}
		\int_{\C} \abs{f(s)}^2 \dif \mu(s) \leq C \norm{f}^2_{H}.
	\end{equation*}

	Finally, we mention the following result on interpolating sequences.
	Recall that a sequence $(s_j)_{j \in \N}$ is called interpolating    for a Hilbert space $H$ of functions analytic on some domain $\Omega$, which admits a reproducing kernel $k_w$ at all $w \in \Omega$, if for all  sequences $(w_j)_{j \in \N}$ satisfying $\sum \abs{w_j}^2/\norm{k_{w_j}}^2 < \infty$ there exists a solution $f \in H$ to the problem $f(s_j) = w_j$. 
	\begin{theo}[Local interpolation theorem \cite{olsen_seip2008}] Let $S = (s_j)$ be a bounded sequence of distinct points in $\C_{1/2}$. Then $S$ is interpolating for $\Hp^2$ if and only if it is interpolating for $H^2(\C_{1/2})$.
 	\end{theo}
	(See   \cite{gordon_hedenmalm1999, mccarthy2004, olofsson2010, saksman_seip2009} for further results on   functions spaces of Dirichlet series.)

	Such precise results are perhaps surprising in view of a deep feature of the theory, which dates back to H.~Bohr \cite{bohr1913lift}. He observed that Dirichlet series can be identified in a natural way with power series of countably infinitely many variables by identifying the $i$'th complex variable $z_i$ with the Dirichlet monomial $p_i^{-s}$,  where $p_i$ is the $i$'th prime number.
Therefore the study of the spaces $\Hp_w$ can be seen as the study of Hilbert spaces of functions in countably infinitely many variables. Namely,
the space $\Hp_w$,  introduced above, is identified with
\begin{equation*}
	 H_w = \left\{ \sum_\nu a_\nu z^\nu : \sum_\nu \abs{a_\nu}^2/w_\nu < \infty \right\}.
\end{equation*}
Here $\nu = (\nu_1, \nu_2, \ldots)$ is a multi-index, $z^\nu = z_1^{\nu_1} z_2^{\nu_2} \cdots$, and we only sum over $\nu$ with 
finite non-zero entries in $\N$.
In particular, the Dirichlet-Hardy space $\Hp^2$ is identified with the Hardy space on the infinite dimensional polydisk, 
$H^2(\D^\infty)$, which corresponds to the choice $w_\nu \equiv 1$. 

For completeness, we briefly discuss  the space  $H^2(\D^\infty)$, or rather, its more  natural counterpart   $H^2(\T^\infty)$, where
\begin{equation*}
	\T^\infty  = \Big\{(z_1, z_2, \ldots ) : z_i \in \T \Big\}
\end{equation*}
is the countably infinite dimensional torus.
$\T^\infty$ is   more natural to work on than $\D^\infty$, since  it is a compact abelian group under coordinate-wise multiplication, and therefore admits a unique normalized Haar measure $\mu$. It follows that we may define the space $L^2(\T^\infty)$ in the usual way.
To define the analytic subspace $H^2(\T^\infty)$, we identify each 
$\chi \in \T^\infty$ with a multiplicative function determined uniquely by $\chi(p_j) = z_j$, where $p_j$ is the $j$'th prime number. The  function $\chi$ is extended to all the positive rational numbers $\Q_+$ by setting $\chi(1/n) = \overline{\chi(n)}$. 
The Fourier spectrum of $f \in L^2(\T^\infty)$ is in this way identified with $\Q_+$. In light of this, we define $H^2(\T^\infty)$ to be the closed subspace whose Fourier spectrum is supported  only on $\N$. Similarly,  for any $p>0$ we   obtain the spaces $H^p(\T^\infty)$. The   Bohr identification now yields a family of spaces Dirichlet-Hardy spaces that we denote by $\Hp^p$. We refer the reader to \cite{bayart2002paper, cole_gamelin1986, hls1997} for further details. In particular, in \cite{cole_gamelin1986}, it is explained how functions in $H^p(\T^\infty)$ can be identified with analytic functions on $\D^\infty \cap \ell^2$, thereby justifying the use the notation $H^p(\D^\infty)$. As a consequence, one direction of the Bohr correspondence can be understood as evaluating a function in $H^p(\T^\infty)$ at the points $(2^{-s}, 3^{-s}, 5^{-s}, \ldots)$ for $\Re s > 1/2$.

Analogues of the three  theorems mentioned above have also been obtained in for the choice of    weights $w_n = (\log n)^\alpha$ \cite{olsen_saksman2010, olsen_seip2008}. However, observe that these spaces, which were introduced by McCarthy  in \cite{mccarthy2004},     correspond to   spaces of functions in infinite variables where monomials of the same degree may differ in norm. 
Our approach in this paper allows us to consider more general choices of weights $w_n$, which correspond to more well-known spaces of infinite variables.
In fact, we are able determine the local behavior of spaces of Dirichlet series corresponding to important   classical spaces. These include the Dirichlet and Bergman spaces on the infinite dimensional polydisk, and the Drury-Arveson space, as well as the general scale of analytic Besov-Sobolev spaces, on the infinite dimensional unit ball. (See examples 1 through \ref{last example} below.)

%
%
%
%
%
%

The structure of the paper is as follows. Our results are presented as theorems 1 to 4 in the following section, where we also include a detailed treatment of the examples mentioned above, as well as a few additional ones. In Section 3 we recall some background results on sampling theory needed in the proofs, and establish  a simple lemma. The proofs of our theorems are   given in sections 4 to 7. In Section 8, we make some concluding remarks.

\section{Results}
We begin with some notation. Recall that we denote the complex variable by $s = \sigma + \im t$, and $\C_{\sigma_0} = \{ \sigma > \sigma_0\}$. In addition, for a bounded interval $I \subset \R$,   we set $\C_I = \{ s \in \C : \im(s-1/2) \notin \R \backslash I \}$. That is, $\C_I$ is the complex plane with two rays on the abscissa $\sigma = 1/2$ removed. Also, we take $f(x) \sim g(x)$ to mean that $g(x)/f(x) \rightarrow 1$ for $x$ approaching some given limit, and by $f(x) \simeq g(x)$ we mean that there exists constants $A,B>0$ such that $Af(x) \leq g(x) \leq Bf(x)$ for all $x$. If only one of the inequalities hold, we use the symbols $\lesssim$ and $\gtrsim$. We denote the Lebesgue measure in the plane by $\dif m$.

Next, we review the definition of the classical scale of spaces $D_\alpha(\C_{1/2})$, which contains the Bergman ($\alpha=-1$), Hardy ($\alpha=0$) and Dirichlet ($\alpha=1$) spaces on the half-plane $\C_{1/2}$.  Accordingly, we set $D_0(\C_{1/2}) := H^2(\C_{1/2})$. For $\alpha <0$, the space $D_\alpha(\C_{1/2})$ consists of the functions $f$ analytic on the half-plane $\C_{1/2}$ and finite in the norm
\begin{equation*}
	 \norm{f}^2_{D_\alpha(\C_{1/2})}  =  \int_{\C_{1/2}} \abs{f(s)}^2 \left( \sigma - \frac{1}{2} \right)^{-\alpha-1} \dif m(s).
\end{equation*}
For $0<\alpha \leq 1$, the space consists of functions analytic on $\C_{1/2}$ for which $f(\sigma) \rightarrow 0$ as $\sigma \rightarrow \infty$, and which are finite in the norm
\begin{equation*}
	 \norm{f}^2_{D_\alpha(\C_{1/2})}  =  \int_{\C_{1/2}} \abs{f'(s)}^2 \left( \sigma - \frac{1}{2} \right)^{-\alpha+1} \dif m(s).
\end{equation*}
The spaces $D_\alpha(\C_{1/2})$ are reproducing kernel spaces. I.e., given $\alpha\leq 1$ and $\xi \in \C_{1/2}$, there exists a function $k_\xi(s)$, such that $\inner{f}{k_\xi} = f(\xi)$. For $\alpha < 1$ these reproducing kernels are now given by
\begin{equation*}
	 k^\alpha_\xi(s) = c_\alpha (s + \bar{\xi} - 1)^{\alpha-1},
\end{equation*}
for the constants $c_\alpha = (-\alpha)2^{-\alpha-1}$ when $\alpha <0$ and $c_\alpha= 2^{\alpha-1}(1-\alpha)^{-1}$ for $0<\alpha<1$. In the limiting case $\alpha=1$, we have
\begin{equation*}
	 k^\alpha_\xi(s) = \frac{1}{\pi} \log \frac{1}{s + \bar{\xi} -1}.
\end{equation*}

To simplify the statements of our theorems, we define the following notions of local embeddings. Here we use the notation $\Omega_I =  (1/2,1] \times I$, where $I \subset \R$ is a bounded interval.
\begin{definition}
	Fix $\alpha \leq 1$. We say the space $\Hp_w$ is locally embedded in the space $D_\alpha(\C_{1/2})$  if
	for every finite interval $I$ there exists a constant $C>0$ such that, if $\alpha <0$ then
	\begin{equation*}
		\int_{\Omega_I}  \abs{F(s)}^2 \left( \sigma - \frac{1}{2} \right)^{-\alpha-1} \dif m(s) \leq C \norm{F}_{\Hp_w}^2,
	\end{equation*}
	if $\alpha = 0$, then
	\begin{equation*}
		\sup_{\sigma > 1/2} \int_{I}  \abs{F(\sigma + \im t)}^2 \dif t \leq C \norm{F}_{\Hp_w}^2,
	\end{equation*}
	and if $0 < \alpha \leq 1$, then
	\begin{equation*}
		\int_{\Omega_I}  \abs{F'(s)}^2 \left( \sigma - \frac{1}{2} \right)^{-\alpha+1} \dif m(s) \leq C \norm{F}_{\Hp_w}^2.
	\end{equation*}
\end{definition}

We can now formulate our first theorem. It generalizes the local embedding theorem mentioned in the introduction.
\begin{theorem} \label{upper theorem}
	Let $(w_n)_{n \in \N}$ be a sequence of non-negative numbers, and $\alpha \in (-\infty,1]$. The following statements are equivalent.
	\begin{itemize}
		\item[(a)] There exists a constant $C>0$ such that for all $x \geq 2$,
		\begin{equation*}
	 		\sum_{n \leq x} w_n \leq C \frac{x}{(\log x)^\alpha}.
		\end{equation*}
		\item[(b)] $\Hp_w$ is embedded locally into the space $D_\alpha(\C_{1/2})$.
	\end{itemize}
\end{theorem}
By analogy to the prime number theorem, the inequality in $(a)$ can be considered as an upper Chebyshev-type inequality.

Although we defer most proofs to later sections, we now give the simplest possible illustration of how   Chebyshev-type inequalities are connected to the
function theoretic properties of the spaces $\Hp_w$. The argument is very similar to the one   in \cite{hls1997}.
\begin{proof}[Proof of $(a) \Rightarrow (b)$ when $\alpha=0$]
For $F \in \Hp_w$ and $\sigma >1/2$, we
calculate by duality
	\begin{equation*} 
		\begin{split}
		 \left(\int_I \Abs{F\left( \sigma + \im t\right)}^2  \dif t \right)^{1/2}
		 &=
		 \sup_{\underset{\norm{g}=1}{g \in L^2}} \int_I F\left( \sigma + \im t\right) g(\im t) \dif t \\
		 &=
		 \sup_{\underset{\norm{g}=1}{g \in L^2}} \sum_{n=1}^N  a_n n^{-\sigma} \int_I g(\im t) n^{-\im t} \dif t \\
		 &=
		 \sqrt{2\pi}
		 \sup_{\underset{\norm{g}=1}{g \in L^2}}  \sum_{n=1}^N a_n \frac{\hat{g}(\log n)}{n^{\sigma}} .
		 \end{split}
	\end{equation*}
	If we multiply and divide by $\sqrt{w_n}$, apply the Cauchy-Schwarz inequality, and take the appropriate limits,
	this yields	 
	\begin{equation} \label{intro: first duality}
		 \left(\int_I \Abs{F\left( 1/2 + \im t\right)}^2  \dif t \right)^{1/2} \lesssim  \norm{F}_{\Hp_w} 
		 		 \sup_{\underset{\norm{g}=1}{g \in L^2(I)}} \underbrace{ \sum_{n\geq1}  \frac{\abs{\hat{g}(\log n)}^2}{n} w_n }_{(*)}.
	\end{equation}
	The functions $\hat{g}$ are   Fourier transforms of   functions  with compact support in a fixed interval in $\R$, which implies
	that they are very regular in the sense that for $\xi \in (k,k+1)$ we get the easy estimate $\abs{\hat{g}(\xi)} \leq \abs{\hat{g}(k)} + \norm{\hat{g}'}_{L^2(k,k+1)}$. This is sufficient to conclude, since by this estimate, the upper Chebyshev inequality for $(w_n)$, and basic properties of the Fourier transform, we obtain
	\begin{equation*}
		 (*) =   \sum_{k=1}^\infty \sum_{n \in (\e^k, \e^{k+1})}  \frac{\abs{\hat{g}(\log n)}^2}{n} w_n
		 \leq  \sum_{k=1}^\infty \frac{\abs{\hat{g}(k)}^2 + \norm{\hat{g}'}_{L^2(k,k+1)}^2}{\e^{k}}  \sum_{n \leq  \e^{k+1}}   w_n
		 \lesssim \norm{g}_{L^2(I)}^2.
	\end{equation*}
\end{proof}

The following result generalizes the theorem on boundary functions mention in the introduction (see also Theorem \ref{lin-type theorem} below), and 
demonstrates the function theoretic significance of lower Chebyshev-type inequalities. 
\begin{theorem} \label{lower theorem}   
	Let $(w_n)$ be a sequence of non-negative numbers and $\alpha \in (-\infty,1]$. If  the upper
	Chebyshev-type inequality of \eqref{chebyshev-type inequalities}    holds      for this choice of $\alpha$ and $(w_n)$, then the following statements are equivalent: 
	\begin{itemize}
		\item[(a)] There exists a constant  such that for all $x \geq 2$,
		\begin{equation*}
	 		\sum_{n \leq x} w_n \gtrsim \frac{x}{(\log x)^\alpha}.
		\end{equation*}
		\item[(b)] For intervals $I$ sufficiently small  and  every $f \in D_{\alpha}(\C_{1/2})$  there exists
		  $F \in \Hp_w$ such that $f-F$ has an analytic continuation across the segment $1/2 + \im I$.  Moreover,
		for every domain $\Gamma$ at a positive distance from $\C \backslash \C_I$, there exists a constant $C$ such that 
		$\norm{f-F}_{L^\infty(\Gamma)} \leq C \norm{f}_{D_\alpha}$.
	\end{itemize}
\end{theorem}
The proof relies in a crucial way on the theory of sampling sequences, and is given in Section \ref{proof of lower theorem}.
See also remarks \ref{lower theorem first remark} and \ref{lower theorem second remark} below on the optimality of this result.

The following result   should be considered an application   of the previous two theorems, and the proof is given in Section \ref{interpolation section}.  
\begin{theorem} \label{interpolation theorem}
	Let $(w_n)$ be a sequence of non-negative numbers and $\alpha \in (-\infty,1]$. If  both the
	Chebyshev-type inequalities \eqref{chebyshev-type inequalities}    hold      for this choice of $\alpha$ and $(w_n)$, then
	the following statements are true.
%
	\begin{itemize}
		\item[(a)] 
				The Carleson measures with compact support for $\Hp_w$ and $D_\alpha(\C_{1/2})$ 
				coincide.
		\item[(b)] The bounded interpolating sequences of $\Hp_w$ and $D_\alpha(\C_{1/2})$ coincide.
	\end{itemize}
\end{theorem}
We state and prove a simple lemma   which is used in the proof of this theorem as it offers a simple application of the 
Chebyshev-type inequalities. The proof of Theorem \ref{interpolation theorem} is given in Section \ref{interpolation section}.
\begin{lemma} \label{reproducing lemma}
		Let $(w_n)$ be a sequence of non-negative numbers and $\alpha \in \R$. If both the   Chebyshev-type 
		inequalities \eqref{chebyshev-type inequalities} hold for this $\alpha$ and $(w_n)$, then for $s = \sigma + \im t$ in $\C_{1/2}$ there are constants   such that
		\begin{equation*}
			 \norm{k_{s}^{D_\alpha}}_{D_\alpha}^2	\lesssim \sum_{n \in \N} w_n n^{-2\sigma} \lesssim  \norm{k_{s}^{D_\alpha}}_{D_\alpha}^2, \qquad \text{as} \quad \sigma \rightarrow 1/2.
		\end{equation*}
\end{lemma}
\begin{proof}
	Denote the $k$-th partial sum of $w_n$ by $W_k$. We sum the left-hand side by parts, and then apply the mean value theorem for $\sigma \in (1/2, 1)$, to get
	\begin{equation*}
	 	\sum_{n \in \N} n^{-2\sigma} w_n = \sum_{n \geq 1} W_n (n^{-2\sigma} - (n+1)^{-2\sigma})
		\simeq \sum_{n \in \N} W_n n^{-2\sigma-1}.
	\end{equation*}
By an application of the Chebyshev-type inequalities, this is seen to be comparable to 
	\begin{equation*}
	 	\sum_{n \in \N} \frac{n^{-2\sigma}}{(\log n+1)^\alpha}.
	\end{equation*}
	The desired conclusion now follows exactly from \cite[Lemma 3.1]{olsen_seip2008}, which gives the behavior of these weighted zeta-type functions as $2\sigma \rightarrow 1$.
\end{proof}

Next, we record a stronger version of Theorem \ref{lower theorem}, as it is more suited for the examples we consider below. The proof is given in Section \ref{proof of lin-type theorem}.
\begin{theorem} \label{lin-type theorem}
	Suppose that for some constant $C>0$ we have
	\begin{equation} \label{lin-type hypothesis}
		\sum_{n \leq x} w_n \sim C \frac{x}{(\log x)^\alpha}, \qquad \text{as} \quad x \rightarrow \infty,
	\end{equation}
	then part $(b)$ of Theorem \ref{lower theorem} holds for every finite interval $I$.
\end{theorem}

Our first example asserts that the above results generalize those mentioned in the introduction.
\begin{example}[The Dirichlet-Hardy space and McCarthy's spaces] \label{dirichlet-hardy and mccarthy example}
	Let $w_n = (1 + \log n)^\alpha$. For $\alpha=0$, we have $\Hp_w = \Hp^2$, and it is trivial to estimate the partial sums.
 So, theorems \ref{upper theorem} and \ref{interpolation theorem} reduce to the local embedding and interpolation theorem, respectively, of the introduction, as well as the statement on the local equivalence of Carleson measures. Note that Theorem \ref{lower theorem} reduces to a   weaker result than the one on local boundary functions in the introduction, while Theorem \ref{lin-type theorem}, which holds in this and all of the following examples, reduces to exactly this theorem. For general $\alpha \leq 1$, we get the same results, except in this case we have to compare the space $\Hp_w$ to   $D_\alpha(\C_{1/2})$. In this case,  the spaces $\Hp_w$ were introduced in \cite{mccarthy2004}, and the corresponding results are contained in \cite{olsen_saksman2010, olsen_seip2008}. Recall that the Dirichlet-Hardy space is identified with $H^2(\D^\infty)$ by Bohr's observation, but     for $\alpha \neq 0$ there is no such natural identification as the monomials $z^n$ and $z^m$ may have different norms even if $\abs{m} = \abs{n}$ for multi-indices $m,n$.
\end{example}
%
%
%
%
%
%
%
Examples \ref{bergman example} through \ref{D infinity example} explore natural analogues on $\D^\infty$ for the scale of spaces on $\D$ which include the Bergman, Hardy and Dirichlet spaces.  To fix notation, we let $f(z) = \sum_{n \in \N} a_n z^n$, and define norms by
\begin{equation} \label{a norm}
	\norm{f}^2_{A_\beta(\D)} = \int_{\D} \abs{f(z)}^2 \dif m_\beta(z) = \sum_{n \in \N} \abs{a_n}^2 \frac{n!}{(\beta+1) (\beta+2) \cdots (\beta + n)},
\end{equation}
where $\dif m_\beta(z)  = ((\beta+1)/\pi) (1 - r^2)^{\beta} r \dif r \dif \theta$ is a probability measure on $\D$ for $\beta > 0$, and
\begin{equation} \label{d norm}
	\norm{f}^2_{D_\alpha(\D)}  =  \sum_{n \in \N} \abs{a_n}^2 (n+1)^\alpha. 
\end{equation}
(Note that the  integral norm in \eqref{a norm} breaks down for $\beta \leq 0$ unless suitably modified. However, for $\beta \in (-1,0)$ we only consider the coefficient norm.)
Here we follow the notation of \cite{hedenmalm_korenblum_zhu2000} and \cite{seip2004book}, respectively. To define the spaces $A_\beta(\D^d)$ and $D_\alpha(\D^d)$ for $d \in \N \cup \{ \infty \}$ while avoiding tedious notation,  we content ourselves in saying that for the space $D_\alpha(\D^d)$,   the monomials $z_1^{\nu_1} \ldots z_d^{\nu_d}$ form  an orthogonal basis with norm the square root of $(\nu_1 + 1)^\alpha \cdots (\nu_d + 1)^\alpha$. For $d=\infty$,   the union of these systems of monomials form the   orthogonal basis. For the spaces $A_\beta(\D^d)$, with $\beta > -1$, the corresponding statements holds in terms of the coefficient norms, while for $\beta >0$ one retains the identity
\begin{equation}  \label{bergman norm}
	\norm{f}_{A_\beta(\D^d)}^2= \int_{\D^d} \abs{f(z_1, \ldots, z_d)}^2 \dif m_\beta(z_1) \ldots \dif m_\beta (z_d). 
\end{equation}

Since 
\begin{equation*}
 	 \frac{(1+ \beta)(2+\beta) \cdots (n + \beta)}{n!} \simeq (1+ n)^\beta,
\end{equation*}
it follows that on the unit disk, or in fact on any finite polydisk, these spaces have equivalent norms for $\alpha = - \beta$ and $\beta > -1$. This no longer holds  on $\D^\infty$.  

%
%
\begin{example}[The spaces $\mathscr{A}_\beta$] \label{bergman example}
	For $\gamma >0$ we define the numbers $d_\gamma(n)$ by the relation $\zeta(s)^\gamma = \sum_{n \in \N} d_\gamma(n) n^{-s}$. By considering the Euler product,
	it is not hard to see that $d_\gamma(p^\nu) = \gamma(\gamma+1) \ldots (\gamma + \nu - 1)/\nu!$. An explicit formula now extends easily to $n \in \N$ since  $d_\gamma(kl) = d_\gamma(k) d_\gamma(l)$ whenever $k$ and $l$ are relatively prime. 
%
%
	We now define   $\mathscr{A}_\beta := \Hp_w$ for the weights $w_n = d_{\beta+1}(n)$.  By the above discussion and the Bohr correspondence, the spaces $\mathscr{A}_\beta$ are isometrically identified with the space $A_\beta(\D^\infty)$. Moreover, 
	the space $\mathscr{A}_\beta$   has translates of $\zeta(s)^\beta$ as its reproducing kernel, and 
	it follows by \cite[Theorem 14.9]{ivic2003book} that for some constant $C>0$, 
	\begin{equation*}
		\sum_{n \leq x} w_n   \sim C {x} (\log x)^\beta.
	\end{equation*}
	Hence,   in the sense of  theorems \ref{upper theorem}, \ref{lower theorem}, \ref{interpolation theorem} and \ref{lin-type theorem},   the space $\mathscr{A}_\beta$ behaves locally like $D_{-\beta}(\C_{1/2})$, as could be expected. 
\end{example}

\begin{example} \label{log zeta example}
	In the limit as $\beta \rightarrow -1^+$, the previous example leads us to also consider the case when $\Hp_w$ is the space of Dirichlet series with   the reproducing kernel given by translates of
	\begin{equation*}
		\log \zeta(s) =  \sum_{p} \sum_{j} \frac{p^{-js}}{j} := \sum_{n \in \N} \frac{\Lambda(n)}{\log n} n^{-s}.
	\end{equation*}
	Here $\Lambda(n)$ is the von Mangoldt function. By a calculation, which   gave   the first proof of the prime number theorem, von Mangoldt \cite{vonmangoldt1895}
	found that 
	\begin{equation*}
		\sum_{n \leq x} \Lambda(n) \sim x. 
	\end{equation*}
	(The partial sum on the left-hand side of this asymptotic formula is usually called the Chebyshev function and is denoted by $\psi(x)$.) It now follows that   the weights $w_n = \Lambda(n)/\log n$ satisfy
	\begin{equation*}
		\sum_{n \leq x}  w_n \sim \frac{x}{\log x},
	\end{equation*}
	whence, by theorems \ref{upper theorem}, \ref{lower theorem}, \ref{interpolation theorem} and \ref{lin-type theorem},   the space $\Hp_w$ behaves locally like $D_{1}(\C_{1/2})$. Observe that by the same arguments, the space $\Hp_{w'}$ with  weights $w'_n = \Lambda(n)$, which has translates of the derivative of $\log \zeta(s)$ as its reproducing kernel,   behaves locally like the Hardy space $H^2(\C_{1/2})$.
\end{example}

%
\begin{example}[The spaces $\mathscr{D}_\alpha$] \label{D beta example}
	Let $d(n)$ denote the number of divisors of the $n$'th integer.
	Explicitly, if $n = p_1^{\nu_1} \cdots p_k^{\nu_k}$, where $p_k$ is the $k$'th prime number and $\nu_k \in \N$, then $d(n) = (\nu_1 + 1) \cdots (\nu_k + 1)$.
	For $\alpha \in \R$, we set $\mathscr{D}_\alpha := \Hp_w$ for the weight $w_n =1/ d(n)^{\alpha}$. 
	As with the weights of the previous two examples, it is very irregular, since highly composite numbers and prime numbers may occur side by side among the natural numbers. 
	%
	%
	%
	Still, it follows by Ramanujan \cite{ramanujan16} and Wilson \cite{wilson22} that there exists a constant $C>0$ such that
	\begin{equation*}
		\sum_{n \leq x} w_n \sim  Cx (\log x)^{2^{-\alpha} -1 }.
	\end{equation*}
	Hence,   by theorems \ref{upper theorem}, \ref{lower theorem}, \ref{interpolation theorem} and \ref{lin-type theorem},   the space $\mathscr{D}_\alpha$ behaves locally like $D_{1 - 2^{-\alpha}}(\C_{1/2})$. 
	\end{example}
	The previous example is surprising as one would expect the local behavior of the space $\mathscr{D}_\alpha$ to correspond to the classical space $D_\alpha(\C_{1/2})$. We remark that in this case, the embedding for $\alpha < 0$ was first observed by Seip \cite{seip2009unpublished}.
\begin{example}[The space $\mathscr{D}_\infty$]\label{D infinity example}
	In the previous example, as $\alpha \rightarrow \infty$, it becomes more difficult for functions of a given norm to have coefficients corresponding to composite numbers. So, as a limit space as $\alpha \rightarrow \infty$, we suggest
	\begin{equation*}
		\mathscr{D}_\infty =  \Set{ \sum_{p \; \mathrm{prime}} a_p p^{-s} : \sum_{p \; \mathrm{prime}} \abs{a_p}^2 < \infty  }.
	\end{equation*}
	In other words, we make the choice of weights
	\begin{equation*}
		w_n = \left\{ \begin{aligned} \; 1  & \quad \text{if} \; n \; \text{is prime}, 
		\\ \; 0 & \quad \text{else}. \end{aligned} \right.
	\end{equation*}
	By the Bohr correspondence, this space is identified with the subspace of $\Hp^2$ 
	spanned by monomials, i.e. 
	\begin{equation*}
		\left\{ \sum_{n \in \N} a_n z_n  :  \sum_{n \in \N} \abs{a_n}^2 < \infty \right\}. 
	\end{equation*}
	By the prime number theorem 
	\begin{equation*}
		\sum_{n \leq x} w_n  \sim \frac{x}{\log x},
	\end{equation*}
	whence we conclude that 
	the space $\mathscr{D}_\infty$ behaves locally like the space $D_1(\C_{1/2})$ in the sense of theorems  \ref{upper theorem}, \ref{lower theorem}, \ref{interpolation theorem} and \ref{lin-type theorem}.
\end{example}
To better see the connection between two previous examples, we consider Figure \ref{graph}. 
\begin{figure}
	\includegraphics{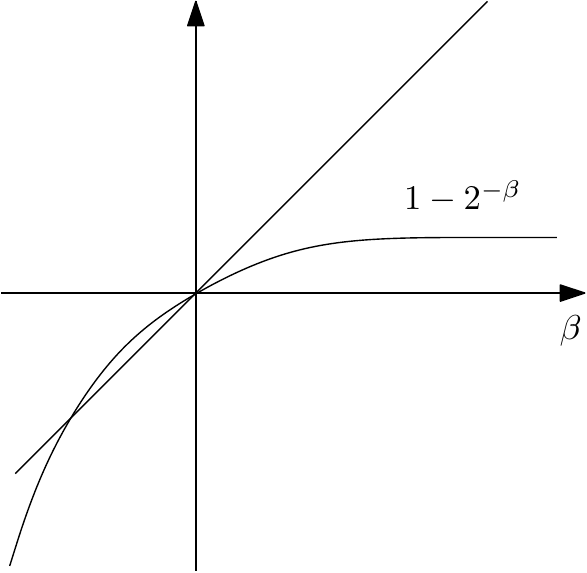}
	\label{graph}
	\caption{} 
\end{figure}
We observe that as the parameter $\alpha$ of example \ref{D beta example} goes to $\infty$ then what we can call the smoothness parameter $1-2^{-\alpha}$ goes asymptotically to $1$. This corresponds to the choice of $\mathscr{D}_\infty$ as an endpoint space, and its local connection to the space $D_1(\C_{1/2})$ appears natural. (See also Section 8.)

Next, we move on to an example on the countably infinite dimensional ball 
\begin{equation*}
	 \mathbb{B}_\infty = \bigg\{ (z_1, z_2, \cdots) \; : \; \sum_{i \in \N} \abs{z_i}^2 < 1 \bigg\}.
\end{equation*}
 In this context, we invoke Ikehara's tauberian theorem explicitly. As indicated in the introduction, it allows one to deduce the behavior of the growth of a sequence by considering functional theoretic properties of a related Dirichlet series. The version we state is due to Delange \cite{delange1954}.
\begin{theo}[Ikehara-Delange]
 	Let $A(x)$ be a non-decreasing function with support in $(0,\infty)$, and for which the function
$F(s) = \int_0^\infty   A(x) x^{-s-1} \dif x$
	converges for $\sigma > \sigma_0 \geq0$. Suppose that $F(s)$ is holomorphic on a neighborhood of the punctured half-plane $\C_{\sigma_0} \backslash \{ \sigma_0 \}$, and that for $\beta < 1$ it holds in this neighborhood   that
	\begin{equation} \label{delange condition}
		F(s) =  g(s) (s - \sigma_0)^{-1+\beta} +  h(s),
	\end{equation}
	for   functions $g, h$ analytic on a neighborhood of $\C_{\sigma_0}$ with $g(\sigma_0) \neq 0$.  Then  as $x \rightarrow \infty$ it follows that
	\begin{equation*}
		A(x)  \sim  {c_\beta} \frac{x^{\sigma_0}}{ \log^{\beta} x}.
	\end{equation*}
	For $\beta = 1$, the same conclusion holds   when \eqref{delange condition} is replaced by 
	\begin{equation*}
		F(s) =  g(s) \log  \frac{1}{s-\sigma_0}  + h(s).
	\end{equation*}
\end{theo}
We remark that this result can be stated in greater generality (see   \cite{korevaar2005}).
Also note that we wish to apply this theorem for $\sigma_0 \neq 1$, in which case the functions in the space $\Hp_w$ will be analytic on $\C_{\sigma_0/2}$. We can still apply theorems \ref{upper theorem} through \ref{lin-type theorem} by considering the shift $F(s - 1/2 + \sigma_0/2)$.
This is tantamount to replacing the weight $(w_n)_{n \in \N}$ by $(n^{1-2\sigma_0} w_n)_{n \in \N}$.
\begin{example}[The Dirichlet-Besov-Sobolev spaces $\mathscr{B}^\gamma_2$] \label{besov example}
	For $\gamma \geq 0$ the classical Besov-Sobolev space on the countably infinite dimensional ball is given as
	\begin{equation*}
		B^\gamma_2(\mathbb{B}_\infty) = \left\{ \sum_{\nu } a_\nu z^\nu :  \sum_{\nu } \frac{ \abs{a_\nu}^2}{\binom{\gamma + \abs{\nu}}{\nu}} < \infty. \right\}.
	\end{equation*}
	Here the multinomial coefficient is defined by
	\begin{equation*}
		\binom{\abs{\nu} + \gamma}{\nu} = \frac{\gamma(\gamma+1) \cdots (\abs{\nu} + \gamma - 1)}{\nu_1 ! \nu_2 ! \ldots },
	\end{equation*}
	when $\gamma >0$ and with $(\abs{\nu}-1)!$ as the denominator in the case that $\gamma = 0$.
	The significance of these coefficients is that the reproducing kernel is given by
	\begin{equation*}
		 K_\gamma(z,w) =  \sum_{\nu }  \binom{\abs{\nu} + \gamma}{\nu} z^\nu \overline{w}^\nu =  \left\{ \begin{array}{rc}  (1 -  \sum z_j \bar{w}_j )^{-\gamma} & \text{if} \; \gamma > 0, \\ \; - \log ({1 - \sum z_j \bar{w}_j } ) & \text{if} \; \gamma=0. \end{array} \right.
	\end{equation*}
	Applying the Bohr correspondence, 
	the space $B^\gamma_2(\mathbb{B}_\infty)$ is seen to
	be isometrically isomorphic to the space of Dirichlet series $\mathscr{B}^\gamma_2$ with the
	reproducing kernel
	\begin{equation*}
		 {k}_\gamma(s,\xi) 
		=  \left\{ \begin{array}{rc}  (1 - \zeta_P(s + \bar{\xi})  )^{-\gamma}  & \text{if} \; \gamma > 0, \\ - \log ( 1 - \zeta_P(s + \bar{\xi})) & \text{if} \; \gamma = 0, \end{array} \right.
	\end{equation*}
	where $\zeta_P(s) = \sum_{p \; \text{prime}} p^{-s}$.
	Let $\rho>1$ be the unique number for which $\zeta_P(\rho) = 1$. 
	Clearly, $\zeta_P$ is analytic on a neighborhood of $\C_{\rho}$, and has a simple zero at $s= \rho$.
	So, by the Ikehara-Delange theorem in combination with theorems  \ref{upper theorem}, \ref{interpolation theorem} and \ref{lin-type theorem}, the space $\mathscr{B}^\gamma_2$ behaves locally like the space $D_{1-\gamma}(\C_{\rho/2})$. In particular we note that the space $\mathscr{B}^1_2$, the Dirichlet-Drury-Arveson space, behaves locally like $H^2(\C_{\rho/2})$.
\end{example}
We mention one last example  that appeared in \cite{mccarthy2004}. It was introduced  as an  example of a Hilbert space of Dirichlet series having a complete Pick kernel.
 Function spaces with this property  
 have attracted interest in the last decade or so. For the definition and related results see e.g.  \cite{agler_mccarthy2002, sawyer2009, seip2004book}.
(The spaces in the previous example have the complete Pick property when $\gamma \in [0,1]$.)
\begin{example}[McCarthy's space]   \label{complete pick example} \label{last example}
	Let $F(n)$ be the number of non-trivial ways to factor the number $n$, counting order. E.g., $F(10) = 3$ since $10, 2 \times 5$ and $5 \times 2$ are the non-trivial ways to factor this number. With this, McCarthy's space is $\Hp_w$ with the weight $w_n = 1/F(n)$. 
	It is straight-forward to check that the reproducing kernel for this space is given by translates of the function
	\begin{equation*}
		 \sum_{n \in \N} F(n) n^{-s} =  \frac{1}{2 - \zeta(s)}.
	\end{equation*}
	From this relation, it follows that $\Hp_w$ has the complete Pick property. If $\rho_1$ is the number satisfying $\zeta(\rho_1) = 2$, it follows as in the previous example, that $\Hp_w$ behaves locally   as   $H^2(\C_{\rho_1/2})$. Note that it is possible to compute weights and determine the local behavior for an entire scale of these spaces by taking the reproducing kernels to some power $\gamma > 0$, or by using the logarithm.  
\end{example}

%
%

\section{preliminaries on sampling sequences and measures}
In this section we give the necessary background on sampling sequences and measures.

\subsection{Sobolev spaces}
To discuss the boundary behavior of functions in spaces which locally behave like $D_\alpha(\C_{1/2})$, we need to introduce the restricted Sobolev spaces $W^\alpha(I)$ and their dual spaces $W^{-\alpha}_0(I)$.

Denote the space of tempered distributions by $\mathcal{S}'(\R)$. For $\alpha \in \R$, we first define the  \emph{unrestricted} Sobolev space 
\begin{equation*}
	 W^\alpha(\R) = \Set{ u \in \mathcal{S}'(\R) : \; \hat u\in L^2_{\mathrm{loc}}\; \mbox{and}\; \| u\|_{W^\alpha}:=\int_\R \abs{\hat{u}(\xi)}^2 (1 + \abs{\xi}^2)^{\alpha} \dif \xi < \infty}.
\end{equation*}
For an open and (possibly unbounded) interval $I \subset \R$, we let $W^\alpha_0(I)$ be the subspace of $W^\alpha(\R)$ that consists of distributions having support in $I$. By a scaling and mollifying argument one easily checks that this subspace coincides with the closure of $\mathcal{C}^\infty_0(I)$ in the norm of $W^\alpha(\R)$. 
With this, we define the  Sobolev space  
\begin{equation*}
	 W^\alpha(I) := W^\alpha(\R)/ W^{\alpha}_0(\R \backslash \bar{I}^C).
\end{equation*}
In other words, the quotient space $W^\alpha(I)$  contains the restrictions of distributions in $W^\alpha(\R)$ to the interval $I$ with the norm
\begin{equation*}
	\norm{u}_{W^\alpha(I)}  =  \inf_{\underset{v|_I = u}{v \in W^\alpha(\R)}} \norm{v}_{W^\alpha(\R)}.
\end{equation*}
Under the natural pairing $(u,v) = \int_\R \hat{u}(\xi) \hat{v}(\xi) \dif \xi$, 
the dual space of $W^\alpha(I)$ is isometric to $W^{-\alpha}_0(I)$, as is readily verified.
It is well-known that the functions  in the spaces $D_{\alpha}(\C_{1/2})$ have distributional boundary values that belong to the Sobolev spaces $W^{\alpha/2}(I)$ on bounded and open intervals $I \subset \R$. 

\subsection{Sampling sequences and measures for the Paley-Wiener space}
Let $H$ be a Hilbert space of functions on some set $\Omega$, with the property that for every $\mu \in \Omega$ there exists a reproducing kernel $k_\mu \in H$.
A sequence $\Lambda = (\lambda_j)$ is a sampling sequence for  $H$ if for all $f \in  H$ it holds
that
\begin{equation*}
	\sum_{\lambda \in \Lambda} \frac{\abs{f(\lambda)}^2 }{ k_{\lambda}(\lambda) } \simeq \norm{f}^2_H.
\end{equation*}
For a bounded interval $I \subset \R$, 
we let $L^2(I)$ denote the subspace of $L^2(\R)$ consisting of functions with support in $I$. The Paley-Wiener space $\PW(I)$ may then be defined as the image of $L^2(I)$ under the Fourier transform. If the interval $I$ is centered at the origin, this space   may also be described
as the space of entire functions of exponential type $\abs{I}/2$ that are square summable on $\R$.

The sampling sequences for $\PW(I)$ on the real line have been given a partial description by Seip \cite{seip1995}  and Jaffard \cite{jaffard1991}, following an idea of Beurling  \cite[p. 345]{beurling1989},  in terms of the   density
\begin{equation*}
	D^-(\Lambda) = \lim_{r \rightarrow \infty} \inf_{\xi \in \R} \frac{\abs{(\xi,\xi+r) \cap \Lambda}}{r}.
\end{equation*}
The result can be stated as follows. 
\begin{theo}[Beurling-Jaffard-Seip]
	Let $\Lambda \subset \R$ be a sequence of distinct numbers. 
	\begin{itemize}
		\item[(i)] If $\Lambda$ is sampling for $\PW(I)$, then $D^-(\Lambda) \geq \frac{\abs{I}}{2\pi}$. 
		\item[(ii)] If $D^-(\Lambda) >\frac{\abs{I}}{2\pi}$, then $\Lambda$ is sampling for $\PW(I)$.
	\end{itemize}
\end{theo}

Following Ortega-Cerda \cite{ortega-cerda1998sampling}, we call a positive measure $\mu$ on $\R$ a sampling measure for $\PW(I)$ if there exists constants  such that for all $g \in \PW(I)$ it holds that
\begin{equation}\label{pw equivalence}
	  \int_\R \abs{g(\xi)}^2 \dif \xi \lesssim 	\int_\R \abs{g(\xi)}^2 \dif \mu(\xi) \lesssim \int_\R \abs{g(\xi)}^2 \dif \xi.
\end{equation}
Note that if the right-hand inequality holds, we   say that $\mu$ satisfies the Carleson-type inequality in \eqref{pw equivalence}, or that $\mu$ is a Carleson measure for $PW(I)$.

Next, given any $r>0$ and $\epsilon >0$ we define
\begin{equation*}
	 \Lambda_\mu(r,\epsilon) = \{ k : \mu([rk,  r(k+1)]) \geq \epsilon \}.
\end{equation*}
The following result \cite[Proposition 1]{ortega-cerda_seip2002} completely characterizes the sampling measures for $\PW(I)$ in
terms of sampling sequences.
\begin{theo}[Ortega-Cerda and Seip 2002]
	Let $I \subset \R$ be an interval centered at $0$. A positive Borel measure $\mu$ is a sampling measure for $\PW(I)$ if and only if:
	\begin{itemize}
		\item[(i)] There exists a constant $C>0$ such that $\mu([\xi,\xi+1)) \leq C$
		for all $\xi \in \R$.
		\item[(ii)] For all sufficiently small $r>0$ there exists a $\delta = \delta(r) >0$ 
		such that $\Lambda_\mu(r,\delta)$ is sampling for $\PW(I)$.
	\end{itemize}
\end{theo}
\begin{remark} \label{sampling theorem remark}
	 Condition $(i)$ of the theorem alone is equivalent to the Carleson-type inequality in \eqref{pw equivalence}
	 (see e.g. \cite{ortega-cerda1998sampling}).
\end{remark}
	Combined with Beurling's density result on sampling sequences,
	this theorem gives a condition on when a measure is a sampling measure (see Corollary \ref{sampling corollary} below).

\subsection{Sampling for weighted spaces}  
We now combine some of the above results to extract a simple condition for sampling in the spaces $\mathcal{F} W^{\beta}_0(I) := \{ \hat{f} : f \in W^\beta_0(I) \}$. To formulate it, we say that a positive Borel measure $\mu$ that satisfies
	\begin{equation} \label{Walpha sampling}
		 \int_\R  \abs{\hat{f}(\xi)}^2 (1+\xi^2)^\beta \dif \xi \lesssim \int_\R \abs{\hat{f}(\xi)}^2 (1+ \xi^2)^\beta \dif \mu(\xi) \lesssim \int_\R  \abs{\hat{f}(\xi)}^2 (1+\xi^2)^\beta \dif \xi,
	\end{equation}
	for all $f \in W^{\beta}_0(I_\epsilon)$, is a sampling measure for $\mathcal{F} W^{\beta}_0(I)$.
\begin{proposition} \label{sampling proposition}
	Let $I$ be some bounded interval, and for $\epsilon>0$,  denote by $I_\epsilon$   the interval co-centric with $I$ such 
	that $\abs{I\backslash I_\epsilon} = 2\epsilon$. 
	If   $\mu$ is a sampling measure for $\mathcal{F} W^{\beta_0}_0(I)$ for some $\beta_0 \in \R$, then for any $\epsilon >0$ and $\beta \in \R$ it holds that $\mu$ is   a sampling measure for $\mathcal{F} W^{\beta}_0(I_\epsilon)$ 
\end{proposition}
\begin{proof}
	Clearly,   if \eqref{Walpha sampling} holds for a given $\beta_0 \in \R$,   
	  we can apply the derivative to extend it to   $\beta_0 + 2\N$. Similarly,  using integration,  $\N$ can be replaced with $\Z$. 
However, to conclude for a general $\beta$, let $h \in C_0^\infty(-\epsilon, \epsilon)$ be a function with a real and positive Fourier transform,  and set 
$\psi_{\beta} = 
		 h \cdot \mathcal{F}^{-1}\{ (1 + \xi^2)^{(\beta - \beta_0)/2} \}$.
	For $f \in C_0^\infty(I_\epsilon)$ it is now readily checked that the function $f \ast \psi_{\beta}$ is in $C_0^\infty(I)$ and  that the estimate
$\abs{\widehat{f \ast \psi_{\beta}}(\xi)}^2 \simeq \abs{\hat{f}(\xi)}^2 (1 + \xi^2)^{\beta   - \beta_0}$
	holds. 
\end{proof}
\begin{remark} \label{sampling proposition remark}
	As follows from the proof, Proposition \ref{sampling proposition} also holds if we only consider   the Carleson-type inequality in \eqref{Walpha sampling}.
\end{remark}
 	By combining the  previous results,  necessary and sufficient  conditions for when a measure is a sampling measure for $\mathcal{F} W^\beta_0(I)$ now follow immediately. To formulate them, we define
	\begin{equation*}
		\Lambda_{\mu}^\beta(r, \delta) = \{ k  : \nu(rk, r(k+1)) \geq \delta (1+ (rk)^2)^\beta \}.
	\end{equation*}
 \begin{corollary} \label{sampling corollary}  
	Suppose that $\nu$ is a sampling measure for $\mathcal{F}W^\beta_0(I)$. Then $\nu(\xi,\xi+1) \leq C(1+\xi^2)^\beta,$
	and given $\epsilon>0$ by choosing sufficiently small $r>0$, there exists $\delta = \delta(r)>0$ such that 
	$D^{-}(\Lambda_\mu^\beta(r,\delta)) \geq {\abs{I_\epsilon}}/{2\pi}$.

	Conversely, suppose that $\nu(\xi,\xi+1) \leq C(1+\xi^2)^\beta$ and that for sufficiently small $r>0$ there exists $\delta=\delta(r)$ such that 
	$D^{-}(\Lambda_\mu^\beta(r,\delta)) > {\abs{I}}/{2\pi}$.
	Then given $\epsilon>0$ the measure $\nu$ is sampling for $\mathcal{F}W^\beta_0(I_\epsilon)$. 
\end{corollary}
\begin{remark} \label{sampling corollary remark}
	By combining remarks \ref{sampling theorem remark} and \ref{sampling proposition remark}, it follows that a positive Borel measure $\nu$ on $\R$ is a Carleson measure on $\mathcal{F}W^\beta_0(I)$ if and only if there exists some $C>0$ such that $\nu(\xi,\xi+1) \leq C(1+\xi^2)^\beta$ for all $\xi \in \R$.
\end{remark}

\subsection{Measures continuous at infinity}
The previous discussion simplifies for measures $\mu$ which are    continuous in the sense that for every $\epsilon > 0$ there exists an $R<\infty$ and $h>0$ such that $\mu([x,x + h]) \leq \epsilon (1 + \xi^2)^{\beta}$ for all $\abs{x} \geq R$. We say that such a measure is $\beta$-continuous at infinity.

The following  theorem is due to Ya.~Lin \cite{lin1965}.
\begin{theo}[Lin]
	Suppose that the positive Borel measure $\mu$ on $\R$  is $0$-continuous at infinity. Then 
	 the measure $\mu$ is sampling for $PW(I)$, for every bounded interval $I \subset \R$,
	if and only if, for some $L>0$ it holds that $\inf_{x \in \R }\{ \mu([x-L,x+L])  \} > 0$.
\end{theo}
By Proposition \ref{sampling proposition}, the above theorem has the following immediate consequence.
\begin{corollary}   \label{lin-type corollary}
	Suppose that the positive Borel measure $\mu$ on $\R$  is $\beta$-continuous at infinity. Then 
	 the measure $\mu$ is sampling for $\mathcal{F}W_0^{\beta}(I)$, for every finite interval $I \subset \R$,
	if and only if  we have $\inf_{\xi \in \R }\{ \mu([\xi-L,\xi+L])  (1+\xi^2)^{-\beta} \}> 0$  for some $L>0$.	
\end{corollary}

%
%

\section{Proof of Theorem \ref{upper theorem}}
We begin with an elementary lemma that will be used in several arguments below.
\begin{lemma} \label{partial sum lemma}
	Let $(w_n)$ be sequence of non-negative numbers, and suppose that $\alpha \in \R$. Then
	\begin{equation*}
		\exists \eta \in (0,1) \: \text{s.t.} \; \sum_{ n \in (\eta x ,x)}w_n \lesssim x(\log x)^\alpha \iff \sum_{n \in (0,x)}w_n \lesssim  x (\log x)^\alpha.
	\end{equation*}
	Moreover, suppose that the upper Chebyshev-type inequality holds, then
	\begin{equation*}
		\exists \eta \in (0,1) \; \text{s.t.} \; \sum_{n\in (\eta x, x)} w_n \gtrsim x (\log x)^\alpha \iff 		\sum_{n \in (0,x)}  w_n \gtrsim x (\log x)^\alpha.
	\end{equation*}
\end{lemma}
\begin{proof}
	It is clear that for each statement, one implication is trivial. As for the `$\implies$' part of the first statement, note that
	\begin{equation*}
		\sum_{n \leq \e^\xi}  w_n = \sum_{k\leq \xi} \sum_{n \in (\e^{k-1}, \e^k)}  w_n \\
		\leq C  \sum_{k \leq \xi}  {{\e^k}}{ k^{-\alpha}} = C \e^\xi \xi^{-\alpha}  \sum_{k \leq \xi}  {\e^{k-\xi}} \left(\frac{k}{\xi}\right)^{-\alpha}.
	\end{equation*}
	This gives the desired conclusion since, by a simple calculation, the sum on the right-hand side is bounded by a constant. 

	Finally,
	`$\Leftarrow$' part of the second statement follows from an argument by contradiction.	
		Indeed,   assume it holds for no $\eta >0$, and set $\psi(x) = x (\log x)^{-\alpha}$. Then there exist sequences $x_k \rightarrow \infty$ 
	and $\eta_k \rightarrow 0$ for which
	\begin{equation*}
		 \sum_{n \in  (\eta_k x_k, x_k) }  w_n \leq  \frac{ \psi(x_k)}{k}.
	\end{equation*}
	Then,  this, and the upper Chebyshev-type inequality, imply
	\begin{equation*}
		\sum_{n \in (1, x_k)}  w_n \leq \sum_{n \in (1, \eta_k x_k)}  w_n + \sum_{n \in (\eta_k x_k, x_k)} w_n 
		\lesssim 
		\psi(\eta_k x_k)  +  \frac{\psi(x_k)}{k}.
	\end{equation*}
	Applying the lower Chebyshev inequality  to the left-hand side now  yields a contradiction,
	since the quotient $\psi(\eta_n x_n)/\psi(x_n)$ goes to zero as $k \rightarrow \infty$.  
\end{proof}
\begin{proof}
	$(a) \Rightarrow (b)$:
	In the introduction we have already proved this for the case $\alpha=0$. For  
	$\alpha\neq 0$, the argument holds with   minor modifications. Namely, if we multiply and divide by $\sqrt{\log^\alpha n w_n}$ on the right-hand side of  \eqref{intro: first duality}, then for $\sigma >1/2$ we obtain
	\begin{multline*}
		 \int_I  \abs{F(\sigma + \im t)}^2 \dif t   \leq \underbrace{ \left( \sum_{n=1}^N \frac{\abs{a_n}^2}{w_n} \frac{1}{(\log n)^\alpha n^{2\sigma -1}} \right) }_{(*)}  \underbrace{\left(  \sup_{\underset{\norm{g}=1}{g \in L^2(I)}} \sum_{n=1}^N \abs{\hat{g}(\log n)}^2 \frac{(\log n)^\alpha w_n}{n} \right) }_{(**)}.
	\end{multline*}
	The factor $(**)$ can be dealt with exactly as before, using the compact support of the functions $g$, to yield $(**) \leq C$. For $\alpha <0$, 
	we use this to evaluate
	\begin{multline*}
		\int_{1/2}^1 \int_I \abs{F(\sigma + \im t)}^2 \left( \sigma - \frac{1}{2}  \right)^{-\alpha-1} \dif t \dif \sigma
		\\ \leq 
		C \sum_{n=1}^N \frac{\abs{a_n}^2}{w_n} \frac{1}{(\log n)^\alpha} \int_{1/2}^1 n^{-(2\sigma-1)} \left(\sigma- \frac{1}{2}\right)^{-\alpha-1} \dif \sigma
		\simeq \norm{F}^2_{\Hp_w}.
	\end{multline*}
	A similar   argument   holds for $\alpha \in (0,1]$.
	
	$(b) \Rightarrow (a)$: Define the function $g_k(s) := \sum_{n \in (\e^k, \e^{k+1})} w_n n^{-s}$. Suppose that  $\alpha <0$. Then the
	local embedding of $\Hp_w$ into $D_\alpha(\C_{1/2})$ implies that for any $\delta>0$ there exists a constant $C>0$ such that
	\begin{equation*}
		 \int_{1/2}^1 \int_{-\delta}^\delta  \abs{g_k(s)}^2 \left(\sigma-\frac{1}{2}\right)^{-1-\alpha} \dif \sigma \dif t
		 \leq
		 C \sum_{n \in (\e^k, \e^{k+1})} w_n.
	\end{equation*}
 	By expanding $\abs{g_k(s)}^2$, we find that the left-hand side of the above expression is equal to
	\begin{equation*}
		2\delta \sum_{n,m \in (\e^k, \e^{k+1})} w_n w_m  \frac{\sin \delta \ln (n/m )}{ \delta \ln (n/m)} \int_{1/2}^{1}  (nm)^{-\sigma} \left( \sigma- \frac{1}{2} \right)^{-1-\alpha} \dif \sigma.
	\end{equation*}
	We fix $\delta>0$ small enough so that $\delta \ln(n/m) \leq \pi/2$. By evaluating the integral with respect to $\sigma$, then up to a constant the previous expression  is seen to be greater than or equal to 
	\begin{equation*}
		  \sum_{n,m \in (\e^k, \e^{k+1})} w_n w_m   \frac{(\log nm)^{\alpha}}{\sqrt{nm}}
		  \geq \frac{(2k)^{\alpha}}{\e^{k+1}} \left( \sum_{n \in (\e^k, \e^{k+1})}  w_n \right)^2.
	\end{equation*}
	By combining the above estimates, we obtain 
	\begin{equation*}
		 \sum_{n \in (\e^k, \e^{k+1})} w_n \lesssim \frac{\e^{k}}{ k^\alpha}.
	\end{equation*}
	By Lemma \ref{partial sum lemma}, 
	this   implies the desired conclusion. The cases $\alpha=0$ and $0 < \alpha<2$ 
	are treated in the same way.
\end{proof}
 
In the paper \cite{olsen_saksman2010}, a more operator theoretic perspective was taken. This made it
possible to study the span of the boundary values of functions in the Dirichlet-Hardy space $\Hp^2$, as well as the more general spaces introduced by McCarthy (see Example \ref{dirichlet-hardy and mccarthy example}). 
In the next section, we utilise this point of view to study the boundary spaces of the spaces $\Hp_w$. To prepare for
this, we introduce the densely defined embedding operator
\begin{equation*}
	E_I : \sum_{n=1}^N a_n n^{-s} \longmapsto \left( \sum_{n=1}^N a_n n^{-\frac{1}{2} - \im t}\right)\bigg|_{t\in I},
\end{equation*}
and establish the following result.
\begin{lemma} \label{upper lemma}  
	Let $(w_n)$ be a sequence of non-negative numbers, and $\alpha \in (-\infty,1]$. Then the conditions
	of the previous theorem are equivalent to either of the statements:
	\begin{itemize}
		\item[(a')] For intervals $I$ of fixed length, there exists a constant such that for all $f \in W^{-\alpha/2}_0(I)$ then
		\begin{equation*}
	 		\sum_{n \in \N} \frac{ \abs{\hat{f}(\log n)}^2}{n} w_n  \lesssim \norm{f}^2_{W^{-\alpha/2}_0(I)}.
		\end{equation*}
		\item[(b')] The operator $E_I$ is bounded from $\Hp_w$ to $W^{\alpha/2}(I)$.
	\end{itemize}
\end{lemma}
\begin{proof}
	The equivalence $(a')\Leftrightarrow (b')$ is obtained exactly as in \cite{olsen_saksman2010}. One simply     computes that the adjoint operator of $E_I$ 
	with respect to the natural non-weighted pairings is given by
	\begin{equation*}
		E_I^\ast : g \in W^{-\alpha/2}_0(I) \longmapsto \sum_{n \in \N}  \frac{\hat{g}(\log n)}{\sqrt{n}} n^{-s} \in \Hp_{1/w}.
	\end{equation*}
	To establish $(a') \Leftrightarrow (a)$, we observe that using the measure
	\begin{equation*}
		\nu = \sum_{n \in \N} \delta_{\log n} \frac{w_n}{n},  
	\end{equation*}
	it is clear that $(a')$ is equivalent to the inequality
	\begin{equation*}
		\int \abs{\hat{f}(\xi)}^2   \dif \nu \lesssim \int \abs{\hat{f}(\xi)}^2 (1 + \xi^2)^{-\alpha/2}  \dif \xi. 
	\end{equation*}
	By Remark \ref{sampling corollary remark}, this is in turn equivalent to $\nu(\xi,\xi+1) \lesssim (1+ \xi^2)^{-\alpha/2}$. It is plain to see that this is equivalent to 
	\begin{equation*}
		\sum_{n \in (\e^{\xi-1},  \e^\xi)} w_n \lesssim {\e^\xi}{\xi^{-\alpha}}, \qquad \forall \xi \geq1.
	\end{equation*}
	By Lemma \ref{partial sum lemma}, this gives the desired conclusion.
\end{proof}

%
%

\section{Proof of Theorem \ref{lower theorem}} \label{proof of lower theorem}
Continuing in the spirit of the   previous lemma,  we take the same approach as in \cite{olsen_saksman2010}.  Namely, 
inspired by the embedding operator $E_I$, which we considered in the previous section, we define an operator 
\begin{equation*}
	 R_I : (a_n)_{n \in \Z^\ast} \longmapsto  \left( \sum_{n \in \N} \frac{a_n n^{-\im t} + a_{-n}n^{\im t}}{\sqrt{n} } \right) \Bigg|_I.
\end{equation*}
Essentially a direct sum of two embedding operators, it allows us the flexibility to study the real parts of functions in $\Hp_w$. To this end, we define  the space
\begin{equation*}
	\ell^2_w(\Z^\ast) = \{ (a_n)_{n \in \Z^\ast} : \sum_{n \in \N} (\abs{a_n}^2 + \abs{a_{-n}}^2)/w_n < \infty \}.
\end{equation*}
With this, we establish the following lemma.
\begin{lemma} \label{onto lemma} 
	Let $(w_n)$ be a sequence of non-negative numbers and $\alpha \in \R$, and suppose that the upper 
	Chebyshev-type inequality for $(w_n)$ holds with this $\alpha$. Then the following are equivalent:
	\begin{itemize}
		\item[(a')] For intervals $I$ sufficiently small, there exists constants such that
		\begin{equation*}
	 		A \norm{f}_{W^{-\alpha/2}_0(I)}^2 \leq \sum_{n \in \N} \frac{\abs{\hat{f}(\log n)}^2 + \abs{\hat{f}(-\log n)}^2}{n}w_n. 
		\end{equation*}
		\item[(b')] For intervals $I$ sufficiently small,   the operator $R_I : \ell^2_w(\Z^\ast) \rightarrow W^{\alpha/2}(I)$ is 
		bounded and onto.
		\item[(c')] For intervals $I$ sufficiently small, then for every $f \in D_{\alpha}(\C_{1/2})$ there exists
		an $F \in \Hp_w$ such that  the real part of $f-F$ vanishes on $1/2 + \im I$, and therefore $f-F$ has an analytic continuation
		across this segment.
	\end{itemize}
\end{lemma}
\begin{proof}
	In light of Lemma \ref{upper lemma}, the equivalence of $(a') \Leftrightarrow (b')$, as well as the implication $(b') \Rightarrow (c')$ is proved more or less exactly as in \cite{olsen_saksman2010}.  The implication $(c') \Rightarrow (b')$ follows by the same line of reasoning. Indeed,
	Let $h(t) \in W^\alpha(I)$ be given, and write $h = u + \im v$. By hypothesis, there exists $F,G \in \Hp_w$ such that $\Re F = u$ and $\Re G = v$ on $1/2  + \im I$. If we write
	\begin{equation*}
		 F = \sum (\alpha_n + \im \beta_n) n^{-s} \qquad \text{and} \qquad G = \sum (\gamma_n + \im \delta_n) n^{-s},
	\end{equation*}
	it follows by considering real parts and imaginary parts, that
	\begin{equation*}
		R_I(c_n) \longmapsto \sum_{n \in \N}  \frac{c_n n^{-\im t} + c_{-n} n^{\im t}}{\sqrt{n}} = h,
	\end{equation*}
	where
	\begin{align*}
		c_n &= \frac{\alpha_n + \gamma_n}{2} + \im \frac{\delta_n + \beta_n}{2},  \\
		c_n &= \frac{\alpha_n - \gamma_n}{2} + \im \frac{\delta_n- \beta_n}{2}.
	\end{align*}
\end{proof}
We are now ready to prove the theorem.
\begin{proof}[Proof of Theorem \ref{lower theorem}] 
	By the previous lemma, it suffices to show $(a) \Leftrightarrow (a')$, $(c') \Rightarrow (b)$ and $(b) \Rightarrow (b')$. 
	
		$(c') \Rightarrow (b)$: In order to get the analytic continuation, the idea is to express the difference $f-F$ by using a Szeg\H o-type integral. The norm estimates then follow in a straight-forward manner. For the Dirichlet-Hardy space, this is proven in detail as a part of \cite[Theorem 1]{olsen_saksman2010}.
		For more general weights satisfying the Chebyshev-type inequalities for $\alpha=0$, this proof can be repeated word by word.
		For $\alpha\neq 0$, the necessary adjustments to the argument are outlined in the proof of  \cite[Theorem 5]{olsen_saksman2010}.

	 $(a) \Leftrightarrow (a')$: We   define the measure
	\begin{equation*}
		\nu = \sum_{n \in \N} \Big( \delta_{\log n}+ \delta_{-\log n} \Big) \frac{w_n}{n}.
	\end{equation*}
		In Lemma \ref{upper lemma}, we established that the upper Chebyshev inequality is equivalent
	to the Carleson-type inequality, so we may assume that it holds. 
	As a consequence,   $(a')$ holds if and only if $\nu$ is a sampling measure for $\mathcal{F}W^{-\alpha/2}_0(I)$.

	Suppose that $\nu$ is a sampling measure for $\mathcal{F}W^{-\alpha/2}_0(I)$. By Corollary \ref{sampling corollary},
	it follows that for some $r>0$ and $\delta >0$, then $\Lambda_{r,\delta}$ has positive density. In particular, for $m >0$ and sufficiently large $\xi$, we have
	\begin{equation*}
		\abs{\Lambda \cap (\xi- rm, \xi) } \geq C > 0. 
	\end{equation*}
	So, for sufficiently large $\xi$ there exists $k \in \N$ with $kr \in (\xi-(m+1)r,\xi -r)$ for which $\nu(kr, (k+1)r) \geq \delta (kr)^{-\alpha/2}$. This yields
	\begin{align*}
		\sum_{  n \in (1, \e^\xi)} w_n 
		 \gtrsim \e^{kr}   \sum_{  n \in (\e^{kr}, \e^{(k+1)r})} \frac{w_n}{n} 
		 \geq \e^{kr} (kr)^{-\alpha/2} 
		\simeq \e^\xi \xi^{-\alpha/2}.
	\end{align*}

		We turn to the converse. By Lemma \ref{partial sum lemma}, if the lower Chebyshev inequality holds, then there exists $\eta \in (0,1)$ 
	so that for large enough $x$ we have
	\begin{equation*}
 		 \sum_{n \in (\eta x, x)} w_n \gtrsim   x (\log x)^{-\alpha}.
	\end{equation*}
	By setting $\eta = \e^{-q}$ and $x = \e^\xi$, this implies that there exists $C>0$ such that for $\xi$ large enough we have
$\nu(\xi-q,\xi) \geq  C \xi^{-\alpha}$.
	In other words, the set
	\begin{equation*}
		\Lambda = \{ k  :   \nu( kq, (k+1)q ) \geq C (1 + (kr)^2 )^{-\alpha/2}  \}
	\end{equation*}
	has density 
$D^{-}(\Lambda)  =   1/q$.
	By basic considerations, it also follows that for $r < q$, the same holds when the constant $C$ is suitably reduced. Hence, by Corollary 	\ref{sampling corollary}, it follows that $\nu$ is a sampling measure for $\mathcal{F}W^{-\alpha/2}_0(J)$, whenever $\abs{J} \leq 2\pi/q$.

  $(b) \Rightarrow (b')$: We use the following basic lemma from operator theory. We refer the reader to    e.g. \cite[Lemma 4, p. 182]{young1980} for a proof.  
\begin{lemma} \label{approximate surjectivity}
	Suppose that $X,Y$ are Banach spaces, and that $Z: X \rightarrow Y$ is a closed linear operator.
	Let $B_X$ and $B_Y$ denote the unit balls of $X$ and $Y$, respectively.
	If there exists a number $M>0$ such that for every $y \in B_Y$ there exists $x \in M B_X$ for which $\norm{Zx - y} \leq 1/2$, then $Z$ is surjective.
\end{lemma}
Suppose that the interval $J$ is such that $R_J$ satisfies part $(b)$ of Theorem \ref{lower theorem}. We show that $(b')$ holds for sufficiently small co-centric intervals $I \subset J$.
To this end, suppose that $h$ in the unit ball of ${W^{\alpha/2}(I)}$ is given, and let $f,g \in D_{\alpha}(\C_{1/2})$ be   functions such that $f + \bar{g}$ has   $h$ as its   boundary distribution on $1/2 + \im I$. Since $h$ is the restriction of a compactly supported distribution, it is not hard to use Laplace transforms to show that $f,g$ can be chosen so that both $\norm{f}_{D_\alpha}$ and $\norm{g}_{D_\alpha}$ are smaller than or equal to  some  constant $B$, independent of $h$. We then 
apply part $(b)$ of Theorem 2 to the disc $\Gamma$  which has the segment $1/2 + \im I$ as a diameter. Since $\Gamma$ has a positive distance to $\C \backslash \C_J$, 
	there exists a constant $C>0$, only depending on $\Gamma$ and $I$,  and $F, G \in \Hp_w$ 
	such that both $\phi := f-F$ and $\psi := g - G$ extends analytically across $1/2+\im I$ and satisfies
	\begin{equation} \label{estimate}
		\sup_{s \in \Gamma} \abs{\phi(s) + \overline{\psi(s)}} \leq C \norm{h}_{W^{\alpha/2}} = C.
	\end{equation}
With this, and a slight abuse of notation, we get
\begin{equation*}
	R_I (F + \overline{G} ) 
	= h  +  (\phi + \overline{\psi})|_I.
\end{equation*}
By a straightforward computation using \eqref{estimate}, we get $\norm{\phi + \overline{\psi}}_{W^{\alpha/2}(I)} \leq C \abs{I}$.
Choosing $I$ so small that $\abs{I} \leq 1/2C$, we   invoke Lemma \ref{approximate surjectivity} to get the desired conclusion.
\end{proof}
We now give two remarks which shows that Theorem \ref{lower theorem} cannot be improved.
\begin{remark} \label{lower theorem first remark} 
	There exist sequences $w_n$ which satisfy both an upper and lower Chebyshev inequality for the same $\alpha$, but for which 
	$\Hp_w$ only reproduces boundary functions on small intervals. Indeed, for each $k \in \N$, choose $n_k \in \N$ such that $\log n \in (k-1/5, k+1/5)$.
	For $n \in \N$, set $w_n = n$ if $n = n_k$ for some $k$, and $w_n = 0$ otherwise. We extend this to negative indices by the rule $w_{-n} = w_n$. On the one hand, it now follows by Kadec's $1/4$ theorem (see e.g. \cite[p. 36, Theorem 14]{young1980})
	that $\nu = \sum \delta_{\log n} w_n/n$ is a sampling measure for $\PW(-\pi,\pi)$. On the other hand, by the Beurling-Jaffard-Seip density theorem above, this sequence cannot be sampling for  $\PW(-\pi - \epsilon, \pi + \epsilon)$ for any $\epsilon>0$.
\end{remark}
\begin{remark} \label{lower theorem second remark}
	As an example of a sequence $w_n$ for which $\Hp_w$ reproduces boundary values, but for which the embedding does not hold. One can choose $w_n$ as  in the   remark above, and set $w'_n = \e^n$ for $n \neq n_k$ for any $k$. Then the space $\Hp_{w''}$, where $w''_n = w_n + w'_n$ reproduces boundary values of $\PW(-\pi,\pi)$ since $\Hp_w \subset \Hp_{w''}$. However, no embedding of the type that we have considered holds.
\end{remark}

%
%

\section{Proof of Theorem \ref{interpolation theorem}} \label{interpolation section}
\begin{proof}[Proof Theorem \ref{interpolation theorem}]
$(a)$: Suppose that  $\mu$ is a Carleson measure for $D_\alpha(\C_{1/2})$ with compact support. Since the space $\Hp_w$ is
embedded into $D_\alpha(\C_{1/2})$ for some $\alpha \in \R$,   it holds that for $N\in \N$ large enough then   $F \in \Hp_w$ implies $F(s)/s^N \in D_\alpha(\C_{1/2})$. Hence,  
\begin{equation*}
	 \int_{\C_{1/2}} \abs{F(s)}^2 \dif \mu(s) \lesssim \int_{\C_{1/2}} \Abs{F(s)/s^N}^2 \dif \mu(s)
	 \lesssim \norm{F/s}_{D_\alpha}^2 \lesssim \norm{F}_{\Hp_w}^2.
\end{equation*}

As for the converse, we use the fact that by Theorem \ref{lower theorem}, for every $f \in D_\alpha(\C_{1/2})$ and interval $I$, there
exists and $F \in \Hp_w$ such that $F -f$ has an analytic extension across the segment $1/2 + \im I$. Hence, by choosing the interval $I$
large enough, we get
\begin{equation*}
	 \int_{\C_{1/2}} \abs{f}^2 \dif \mu \leq 	 \int_{\C_{1/2}} \abs{F}^2 \dif \mu + \int_{\C_{1/2}}\abs{f - F}^2 \dif \mu 
	 \lesssim \norm{F}_{\Hp_w}^2 + C.
\end{equation*}
The conclusion now follows either by the norm control offered by Theorem \ref{lower theorem}, or the closed graph theorem.

$(b)$:  
The following argument is different from the one found in \cite{olsen_seip2008}, which was applied in the case of the weights $w_n = (\log n)^\alpha$, as it avoids use of the reproducing kernels beyond their role as norms for point evaluations.  
Instead, it relies on Lemma \ref{approximate surjectivity} and part $(b)$ of Theorem \ref{lower theorem}.
	One direction is simple, and follows by the same arguments as in \cite{olsen_seip2008}. Indeed, by Theorem \ref{upper theorem}, 
	the space $\Hp_w$ is embedded locally into the space $D_\alpha(\C_{1/2})$.   As above, it follows
	that if $F \in \Hp_w$, then for some $N \in \N$ large enough, we have $F/s^N \in D_\alpha(\C_{1/2})$. Since we are dealing with bounded interpolating sequences, the problem $F(s_n ) =  w_n$ has a solution if and only if $F(s_n) = w_n s_n^N$ does. Hence, $f(s) = F(s)/s^N$ 
	is a function in $D_\alpha(\C_{1/2})$ that solves the problem $f(s_n) = w_n$.

	We turn to the converse. Assume that $S = (s_n)_{n \in \N}$ is a bounded interpolating sequence for $D_\alpha(\C_{1/2})$, and write
	$k_n^{D_\alpha}$ and $k_n^{\Hp_w}$ for the reproducing kernels at $s_n$ of the respective spaces.
	This means that the interpolation operator defined by
	\begin{equation*}
		 f \in D_\alpha(\C_{1/2}) \longmapsto \left( \frac{f(s_n)}{\norm{k_n^{D_\alpha}}_{D_\alpha}} \right)_{n \in \N} \in \ell^{2}
	\end{equation*}
	is bounded and onto $\ell^{2}$. 
	Since $\norm{k_n^{D_\alpha}}_{D_\alpha} \simeq \norm{k_n^{\Hp_w}}_{\Hp_w}$, as follows from Lemma \ref{reproducing lemma},  the same remains true
	if we replace the weights of the operator by $1/\norm{k_n^{\Hp_w}}_{\Hp_w}$. 
	Next, without loss of generality, we may assume that the sequence $\seq{s_n}_{n \in \N}$ satisfies $\sigma_{n+1} \geq \sigma_n$. With this in mind we set $S_N = \seq{s_n}_{n\geq N}$ and define the   operators
	\begin{equation*}
		\mathcal{T}_N : F \in \Hp_w \longmapsto \left( \frac{F(s_n)}{\norm{k_n^{\Hp_w}}_{\Hp_w}} \right)_{n \geq N}.
	\end{equation*}
	
 		By the same reasoning as in the proof of  \cite[Thm. 2.1]{olsen_seip2008}, it follows that if $\mathcal{T}_N$ is surjective for some $N \in \N$, then the operator $\mathcal{T}_1$ is also surjective. The argument uses Lagrange-type sums   of  finite products of the type
	\begin{equation*}
		\prod_{j=1}^N \left( 1 - p_j^{-(s-s_j)} \right),
	\end{equation*}
	where the $p_j$ are prime numbers.

	 Next, we   show that $\mathcal{T}_N$ is onto for large enough $N$. So, suppose that $b_j$ is in the unit ball of $\ell^2$, and assume first that 
	there exists $f$ in the unit ball of $D_\alpha(\C_{1/2})$, such that $f(s_j) = b_j\norm{k_j^{\Hp_w}}_{\Hp_w}$ for $j \in \N$. (In general, it only follows by the open mapping theorem that such an $f$ exists in some dilation of the unit ball.)

	To apply part $(b)$ of Theorem  \ref{lower theorem}, let $\Gamma$ be an open disk in $\C_{1/2}$ that contains $S$, and let
	$I\subset \R$ be some bounded interval such that $\sup \set{ \abs{\Im s} : s \in I } \geq 2 \sup \set{ \abs{\Im s} : s \in \Gamma}$.
	Now, since $\Gamma$ is at a positive distance from $\C \backslash \C_I$,  
	there exist a constant $C>0$, only depending on $\Gamma$ and $I$,  and $F \in \Hp_w$ 
	such that $\phi := f-F$ extends analytically across $1/2+\im I$ and satisfies
	\begin{equation*}
		\sup_{s \in \Gamma} \abs{\phi(s)} \leq C \norm{f}_{D_\alpha(\C_{1/2})}.
	\end{equation*}
	It now follows, with a slight abuse of notation, that
	\begin{equation*}
		 \mathcal{T}_N F(n) = \mathcal{T}_N f(n) +  \mathcal{T}_N (F-f)(n)
		 = b_n + \phi(s_n)/\norm{k_n^{\Hp_w}}_{\Hp_w}.
	\end{equation*}
	So, to conclude by  Lemma \ref{approximate surjectivity},  we need to show that for $N$ large enough, we have 
	\begin{equation} \label{interpolation: size of KN}
		 \Norm{\left(\frac{\phi(s_n)}{\norm{k_n^{\Hp_w}}_{\Hp_w}}\right)_{n \geq N}}_{\ell^{2}} \leq 1/2.
	\end{equation}
	But this follows immediately as $\phi(s_n)$ is uniformly bounded, and the sequence $(1/\norm{k_n^{\Hp_w}}_{\Hp_w})_{n \in \N}$ is square summable. 
\end{proof}

%
%

\section{Proof of Theorem \ref{lin-type theorem}}  \label{proof of lin-type theorem}
Let $I$ be any bounded interval in $\R$. In light of part $(a')$ of Lemma \ref{onto lemma} it suffices to show that there exist constants such that  for all $f \in C^\infty_0(I)$ we have
\begin{equation*}
	 \norm{f}_{W^{-\alpha/2}_0(I)}^2 \lesssim \sum_{n \in \N} \frac{\abs{\hat{f}(\log n)}^2 + \abs{\hat{f}(-\log n)}^2}{n}w_n
	\lesssim  \norm{f}_{W^{-\alpha/2}_0(I)}^2.
\end{equation*}
By definition this is equivalent to the measure
\begin{equation*}
	\mu  =   \sum_{n \in \N} \frac{\delta_{\log n} + \delta_{-\log n}}{n} w_n
\end{equation*}
being a sampling measure for $\mathcal{F}W_0^{-\alpha/2}$. To apply Corollary \ref{lin-type corollary}, we first need to check that $\mu$ 
is $(-\alpha/2)$-continuous at infinity.

For $L>0$ and $\xi>0$ we get
\begin{equation*}
	\mu[(\xi-L, \xi)]  = \sum_{n \in (\e^{\xi-L}, \e^\xi)} \frac{w_n}{n} 
	\lesssim \e^{-\xi} \bigg(  \sum_{n \leq \e^{\xi}} w_n- \sum_{n \leq \e^{\xi-L}} w_n \bigg).
\end{equation*}
Given $\epsilon>0$, we choose $\xi$ large enough for \eqref{lin-type hypothesis} to yield
	$(C+\epsilon) \xi^{-\alpha} - (C-\epsilon) \e^{-L} (\xi-L)^{-\alpha}$.
Clearly, by choosing $L>0$ small, and letting $\xi$ be large enough, we obtain the desired inequality
\begin{equation*}
	\mu[(\xi-L,\xi)] \leq \epsilon ( 1 + \xi^2 )^{-\alpha/2}. 
\end{equation*}

To complete the proof, we need to check that there exists some $L>0$ such that
\begin{equation*}
	\inf_{\xi \in \R} \mu[(\xi-L, \xi)] (1 + \xi^2)^{\alpha/2}  >0.
\end{equation*}
But by what was done in the proof of $(a) \Leftrightarrow (a')$ in Theorem \ref{lower theorem},
there exists $L>0$ such that
 for large enough $\xi$ we have
\begin{equation*}
 	 \sum_{n \in ( \e^{\xi-L}, \e^{\xi})} w_n \gtrsim  \e^\xi  (\xi - L)^{-\alpha}.
\end{equation*}
By Lemma \ref{partial sum lemma}, the conclusion now follows.

\section{further remarks}
It is possible to define the spaces $\mathscr{A}_\beta$ of example \ref{bergman example} when $\beta >0$ for general $p \neq 2$ 
using the expression \eqref{bergman norm}. By \cite{cole_gamelin1986}, the resulting function spaces of Dirichlet series have bounded point evaluations on $\C_{1/2}$. That the same is true for the spaces $\mathscr{D}_\alpha$ of example \ref{D beta example}, for $\alpha <0$, is less obvious. However, as it is possible to solve the Hausdorff moment problem $(n+1)^{\alpha} = \int_0^1 r^{2n+1} \dif \nu_\alpha$, for some positive measure $\nu_\alpha$, one obtains a radial probability measure on $\D$ (see e.g. \cite[Chapter III]{widder1941}). This yields the required integral expression for the norm on polydisks.

By the previous remark, it is not hard to determine the multiplier algebras of these spaces. In the language of \cite{cole_gamelin1986}, it is clear that the multipliers of the spaces $A_\beta(\D^\infty)$ and $D_\alpha(\D^\infty)$ are exactly the elements of  the spaces   $H^\infty$   for the respective  infinite product   measures. But as these measures are products of radial probability measures supported on $\bar{\D}$, it was shown in  \cite[Theorem 11.1]{cole_gamelin1986} that these spaces are simply $H^\infty(\T^\infty)$.  As explained in \cite{hls1997} for the space $\Hp^2$, which we identified with the space  $H^2(\T^\infty)$ in the introduction, it now follows  that the multiplier algebra of both the spaces $\mathscr{A}_\beta$ and $\mathscr{D}_\alpha$ is
\begin{equation*}
	\Hp^\infty = \Big\{ \sum a_n n^{-s} : \sup_{\Re s > 0} \abs{\sum a_n n^{-s}  } < \infty \Big\}.
\end{equation*}
The same argument holds for any $p \geq 1$. Recently, similar results were   obtained  for $p \in (0,1)$ for function spaces on finite polydisks by Harutyunyan and Lusky \cite{harutyunyan_lusky2009}.


Our next remark concerns a consequence of an improvement of an inequality of Hardy and Littlewood. Mateljevic \cite{mateljevic1979} showed that  the constant  $C=1$ is best possible in  the inequality 
\begin{equation}  \label{one dimension}
	\sum_n \abs{a_n}^2  (n+1) \leq C \int_\T \abs{f(\e^{\im t})} \frac{\dif t}{\pi}.
\end{equation}
We remark that the proof of the latter fact was essentially known in the smooth case to Carleman, who considered only the finite Blaschke products, and was proved in full generality by Mateljevic  using the same method. Since it seems that his paper did not become widely known, the same proof was later rediscovered by Vukotic   \cite{vukotic2003}. In language of Dirichlet series, Helson \cite{helson2006studia} exploited this precise estimate  to prove, using a method due to Bayart, that
\begin{equation*}
	\norm{F}_{\mathscr{D}_{-1}}  \leq \norm{F}_{\Hp^1}.
\end{equation*}
Our observation is that by following the classical proof of the Riesz-Thorin interpolation theorem, it is possible to interpolate between \eqref{one dimension} and the Plancherel identity for $p=2$  to obtain (in  the notation of example \ref{D beta example})
\begin{equation} \label{interpolated guy}
	\norm{f}_{D_{1 - 2/p}(\D)} \leq \norm{f}_{H^p(\D)}.
\end{equation}
By generalising the argument of Bayart and Helson, this yields
\begin{equation*}
	\norm{F}_{\mathscr{D}_{1-2/p} } \leq \norm{F}_{\Hp^p}.
\end{equation*}
With respect  to Figure \ref{graph}, this family of inequalities takes place between the two points of intersection between the curves 
which represent the  "smoothness" of the spaces of $D_\alpha(\C_{1/2})$ and $\mathscr{D}_\alpha$
I.e., at the points $\alpha=-1$ and $\alpha=0$,   where the space $\mathscr{D}_\alpha$ behaves locally as one would expect. 

In addition to the local embeddings discussed above,   others are possible. For instance,   Seip observed that it follows from \eqref{interpolated guy} and a duality argument that   $\mathcal{D}_\alpha$ is locally embedded  into the space $H^{2^{\alpha+1}}(\C_{1/2})$. Specifically, given a bounded interval $I$, then there exists a constant $C>0$ such that for $f \in \mathscr{D}_\alpha$ we have
\begin{equation*}
	 \sup_{\sigma > 1/2} \int_I \abs{f(\sigma + \im t)}^{2^{\alpha+1}} \leq C \norm{f}_{\mathscr{D}_\alpha}^{2^{\alpha+1}}.
\end{equation*}
We point out the the best possible constant of \eqref{interpolated guy} is not needed for this argument.


Finally, we mention that the Helson-Bayart inequality mentioned above  is used in  \cite{helson2006studia} to prove a special case of the Nehari lifting theorem for Hankel forms on the Hardy space $H^2(\T^\infty)$. A Hankel form in countably infinitely many variables is defined by
\begin{equation*}
	(a_j, b_j) := \sum_{j,k \in \N} a_j b_{k} \rho_{jk},
\end{equation*}
where $j$ and $k$ are multiplied in the index of $\rho_{jk}$. (Note that the one variable definition is retrieved by only  summing over indices $j= 2^m$.) The result of Helson says that if the Hankel form is a Hilbert-Schmidt operator, then there exists a function $\phi$ in $L^\infty(\T^\infty)$ such that $\hat{\phi}(n) = \rho_n$ for $n \in \N$. The connection is that the Hilbert-Schmidt condition is exactly
\begin{equation*}
	\sum_{j,k \in  \N} \abs{\rho_{jk}}^2 =  \sum_{n\in \N} d(n) \abs{\rho_n}^2 < \infty,
\end{equation*}
where $d(n)$ is the number of divisors function (see Example \ref{D beta example}).
By  the Helson-Bayart inequality, the solution now follows by a duality argument. In the general case, the problem has been settled by   Ferguson and Lacey on   the bidisk  \cite{ferguson_lacey2002} and Lacey and Terwilleger on polydisks of finite dimension \cite{lacey_terwilleger2009}, but it remains open on the infinite dimensional polydisk. (See also \cite[p. 54]{helson2005book} for a discussion of this problem.)

\section*{Acknowledgements}
Parts of this paper is based on research done during the work on the  PhD thesis of the author, and he would therefore like to thank his supervisor   professor Kristian Seip for   advice and access to the unpublished note \cite{seip2009unpublished}.
The author would also like to thank professor Eero Saksman for valuable conversations regarding  the proof of Theorem \ref{interpolation theorem}, and Anders Olofsson for suggesting example 2.

\bibliographystyle{amsplain}
\def\cprime{$'$} \def\cprime{$'$} \def\cprime{$'$}
\providecommand{\bysame}{\leavevmode\hbox to3em{\hrulefill}\thinspace}
\providecommand{\MR}{\relax\ifhmode\unskip\space\fi MR }
\providecommand{\MRhref}[2]{%
  \href{http://www.ams.org/mathscinet-getitem?mr=#1}{#2}
}
\providecommand{\href}[2]{#2}

\end{document}